\documentclass[reqno,11pt]{amsart}
\usepackage[foot]{amsaddr}
\usepackage{bbold}
\usepackage{charter}
\usepackage[margin=1in]{geometry}
\usepackage[colorlinks=true,linktoc=all,linkcolor=blue,citecolor=blue]{hyperref}

\theoremstyle{plain}
\newtheorem{theorem}{Theorem}[section]
\newtheorem*{theorem*}{Theorem}
\newtheorem{proposition}{Proposition}[theorem]
\newtheorem{lemma}[theorem]{Lemma}
\newtheorem{corollary}[theorem]{Corollary}
\theoremstyle{definition}

\newtheorem*{definition*}{Definition}

\newtheorem{conjecture}[theorem]{Conjecture}
\numberwithin{equation}{section}
\everymath{\displaystyle}
\allowdisplaybreaks

\newcommand{\B}{\mathbb{B}}
\newcommand{\C}{\mathbb{C}}
\newcommand{\D}{\mathbb{D}}
\newcommand{\N}{\mathbb{N}}
\newcommand{\R}{\mathbb{R}}
\newcommand{\1}{\mathbb{1}}
\newcommand{\Bloch}{\mathcal{B}}
\newcommand{\lilBlochDist}{\Bloch_{0^*}}
\newcommand{\Hinfu}[1]{H^\infty_{\hspace{-.3ex}#1}}
\newcommand{\Hinfmu}{\Hinfu{\mu}}
\newcommand{\Log}{\mathrm{Log}}
\newcommand{\conj}[1]{\overline{#1}}
\newcommand{\modu}[1]{\left|#1\right|}
\newcommand{\distbd}[1]{\partial^*\hspace{-.4ex} #1}
\newcommand{\Aut}{\mathrm{Aut}}

\newcommand{\norm}[1]{\left|\left|#1\right|\right|}
\newcommand{\supnorm}[1]{\norm{#1}_\infty}
\newcommand{\blochnorm}[1]{\norm{#1}_\Bloch}
\newcommand{\munorm}[1]{\norm{#1}_{\Hinfmu}}
\newcommand{\wcompop}[2]{W_{#1,#2}}
\newcommand{\wco}{\wcompop{\psi}{\varphi}}

\newcommand{\qte}[1]{``#1"}
\newcommand{\listitem}[1]{\textup{(}#1\textup{)}}


\title[Weighted Composition Operators]{Weighted Composition Operators from the Bloch Space to\\ Weighted Banach Spaces on Bounded Homogeneous Domains}

\author{Robert F.~Allen}
\address{Department of Mathematics, University of Wisconsin-La Crosse}
\email{rallen@uwlax.edu}

\begin{document}

\begin{abstract}
We study the bounded and the compact weighted composition operators from the Bloch space into the weighted Banach spaces of holomorphic functions on bounded homogeneous domains, with particular attention to the unit polydisk.  For bounded homogeneous domains, we characterize the bounded weighted composition operators and determine the operator norm.  In addition, we provide sufficient conditions for compactness.  For the unit polydisk, we completely characterize the compact weighted composition operators, as well as provide \qte{computable} estimates on the operator norm.
\end{abstract}

\maketitle

\section{Introduction}
Let $X$ and $Y$ be Banach spaces of holomorphic functions on a domain $\Omega \subset \C^n$.  For holomorphic functions $\psi:\Omega \to \C$ and $\varphi:\Omega\to\Omega$, the \textit{weighted composition operator} from $X$ to $Y$ with symbols $\psi$ and $\varphi$ is defined as $$\wco f = \psi(f\circ\varphi),$$ for $f \in X$.  The weighted composition operator is the generalization of the \textit{multiplication operator} $M_\psi f = \psi f$ and the \textit{composition operator} $C_\varphi f = f\circ\varphi$, called the component operators.

The study of weighted composition operators on the Bloch space of the unit disk $\D$ began with the work of Ohno and Zhao \cite{OhnoZhao:2001}, where the boundedness and compactness were characterized.  In higher dimensions, these operators on the Bloch space have been studied on the unit polydisk by Chen, Stevi\'c, and Zhou \cite{ChenStevicZhao}, and on bounded homogeneous domains by the author and Colonna \cite{AllenColonna:2010}.

The bounded and compact weighted composition operators from the Bloch space to $H^\infty$ were characterized by Ohno \cite{Ohno:2001} and Hosokawa, Izuchi, and Ohno \cite{HosokawaIzuchiOhno:2005} in the one-dimensional case, and by Li and Stevi\'c \cite{LiStevic:2007-II} in the case of the unit ball.  In \cite{AllenColonna:2010}, the author and Colonna characterized the boundedness, determined the operator norm, and gave a sufficient condition for compactness in the case of a general bounded homogeneous domain in $\C^n$.

Composition and multiplication operators on $\Hinfmu(\D)$ were first studied by Bonet, Doma\'{n}iski, Lindstr\"{o}m and Taskinen \cite{BonetDomaniskiLindstromTaskinen:1998}, and then later by Bonet, Doma\'{n}iski and Lindstr\"{o}m \cite{BonetDomaniskiLindstrom:1999,BonetDomaniskiLindstrom:1999-II}.  The weighted composition operators on these spaces were studied by Contreras and Hern\'{a}ndez-D\'{i}az \cite{ContrerasHernandezDiaz:2000}, and by Galindo and Lindstr\"{o}m \cite{GalindoLindstrom:2010}.

In \cite{Stevic:2008}, Stevi\'c determined the norm of the bounded weighted composition operators from the Bloch space to the weighted Banach space $\Hinfmu$ of the unit ball by means of the more general $\alpha$-Bloch spaces.  Likewise, Yang characterized the bounded and compact weighted composition operators from the Bloch space to $\Hinfmu$ on the unit ball via the $\alpha$-Bloch space \cite{Yang:2009}.  Zhu characterized the weighted composition operators from the Bloch space to $\Hinfmu$ on the unit ball by means of studying the $F(p,q,s)$ spaces, for which the Bloch space is a special case \cite{Zhu:2009}.  These operators have not previously been studied for spaces of functions defined on the unit polydisk.

A difficulty in the study of such operators on domains in several variables is the fact that function theory on the unit ball $\B_n$ and the unit polydisk $\D^n$ are vastly different.  In recent years, work has been done to generalize these two spaces in an effort to consolidate work.  The bounded homogeneous domains are a natural generalization of both $\B_n$ and $\D^n$, and spaces of holomorphic functions on such domains have been studied.  In this paper, we generalize the study of weighted composition operators from the Bloch space to the weighted Banach spaces from $\B_n$ and $\D^n$ to bounded homogeneous domains.  This unifies the study of such operators, and allows for study these operators on spaces that are not just $\B_n$ and $\D^n$, examples of which can be can be found in \cite{Vesentini:1967}.

Currently the study of these operators in higher dimensions has taken place in the setting of the unit ball $\B_n$.  This paper introduces to the literature, the study of these operators in the unit polydisk setting.  As is normally the case, the techniques used are different than those used for the unit ball.  In addition, the author and Colonna studied weighted composition operators from the Bloch space into $H^\infty$ on bounded homogeneous domains in \cite{AllenColonna:2010}.  This paper extends this work to the weighted Banach spaces $\Hinfmu$.

\subsection{Organization of the Paper}
In Section \ref{Section:Preliminaries}, we define the Bloch space and weighted Banach spaces $\Hinfmu$ on a bounded homogeneous domain as well as collect useful facts about Bloch functions on such domains.  In Section \ref{Section:Boundedness}, we characterize the bounded weighted composition operator from $\Bloch(D)$ into $\Hinfmu(D)$.  In addition, we determine the operator norm of such operators.  In Section \ref{Section:Compactness}, we develop sufficient conditions for the bounded weighted composition operators $W_{\psi,\varphi}:\Bloch(D) \to \Hinfmu(D)$ to be compact.  In Section \ref{Section:Polydisk}, we characterize the bounded and the compact weighted composition operators, as well as provide \qte{computable} estimates on the operator norm.  Finally, in Section \ref{Section:Conclusions} we end with some concluding thoughts and open problems.

\section{Preliminaries}\label{Section:Preliminaries}
Let $D$ be a domain in $\C^n$.  We denote by $H(D)$ the set of holomorphic functions from $D$ into $\C$, $S(D)$ the set of holomorphic self-maps of $D$, and by $\Aut(D)$ the set of biholomorphic maps of $D$. The space $H^\infty(D)$ of bounded holomorphic functions on $D$ is a Banach algebra equipped with norm $\supnorm{f} = \sup_{z \in D}\; \modu{f(z)}$.

A domain $D$ is \textit{homogeneous} if $\Aut(D)$ acts transitively on $D$.  Every homogeneous domain is equipped with a canonical metric, called the \textit{Bergman metric}, invariant under the action of $\Aut(D)$ \cite{Helgason:1962}.

For a continuous, strictly positive function $\mu:D \to \R_+$, called a \textit{weight}, the \textit{weighted Banach spaces} $\Hinfmu(D)$ is defined as $$\Hinfmu(D) = \left\{f \in H(D) : \sup_{z \in D}\; \mu(z)\modu{f(z)} < \infty\right\},$$ which is a Banach space under the norm $\munorm{f} = \sup_{z \in D}\; \mu(z)\modu{f(z)}$.  Notice that if $\mu \equiv 1$ on $D$, then $\Hinfmu(D) = H^\infty(D)$.

In \cite{Timoney:1980} and \cite{Timoney:1980-I}, Timoney defined and studied the Bloch functions on a bounded homogeneous domain. In this paper, we conform to his notation, as follows.

Let $D$ be a bounded homogeneous domain. For $z \in D$ and $f \in H(D)$, define $$Q_f(z) =
\sup_{u \in \C^n\setminus\{0\}}\; \frac{\modu{\nabla(f)(z)u}}{H_z(u,\conj{u})^{1/2}},$$ where $\nabla(f)(z)$ is the gradient of $f$ at $z$, for $u=(u_1,\dots,u_n)$, $$\nabla(f)(z)u = \sum_{k = 1}^n \frac{\partial f}{\partial z_k}(z)u_k,$$ and $H_z$ is the Bergman metric on $D$ at $z$.  For the definition of the Bergman metric (also known as the Poincar\'{e} metric or distance), see \cite{Krantz:2000} Definition 1.4.14.  The Bergman metric for the unit polydisk $\D^n$ is defined as
\begin{equation}
\nonumber
H_z(u,\conj{v}) = \sum_{j=1}^n \frac{u_j\conj{v_j}}{(1-\modu{z_j}^2)^2},
\end{equation} where $u,v \in \C^n$ and $z \in \D^n$.

The \textit{Bloch space} $\Bloch(D)$ on a bounded homogeneous domain $D$ is the set of all functions $f \in H(D)$ for which $$\beta_f = \sup_{z \in D}\; Q_f(z) < \infty.$$  Timoney proved that $\Bloch(D)$ is a Banach space under the norm $$\blochnorm{f} = \modu{f(z_0)} + \beta_f,$$ where $z_0$ is some fixed point in $D$ \cite{Timoney:1980}. For convenience, we shall assume throughout that $0 \in D$ and choose $z_0 = 0$.

The {\it $*$-little Bloch space of $D$} is the subspace of $\Bloch(D)$ defined as $$\Bloch_{0*}(D) = \left\{f \in \Bloch(D) : \lim_{z \to \distbd{D}} Q_f(z) = 0\right\},$$ where $\distbd{D}$ is the distinguished boundary of $D$. 

In \cite{Timoney:1980}, Timoney also proved that $H^\infty(D)$ is a subspace of $\Bloch(D)$ and for each $f \in H^\infty(D)$, $\blochnorm{f} \leq |f(0)|+ c\supnorm{f}$ where $c$ is a constant depending only on the domain $D$.  

In the rest of the section, we collect useful results of Bloch functions on bounded homogeneous domains.  For a bounded homogeneous domain $D$, and $z \in D$, define
$$\begin{aligned}
\omega(z) &= \sup\;\{|f(z)| : f \in \Bloch(D), f(0) = 0 \text{ and } \norm{f}_\Bloch \leq 1\},\\
\omega_0(z) &= \sup\;\{|f(z)| : f \in \lilBlochDist(D), f(0) = 0 \text{ and } \norm{f}_\Bloch \leq 1\}.
\end{aligned}$$

\begin{lemma}[Lemma 4.1 of \cite{AllenColonna:2010}]\label{omega inequality} Let $D$ be a bounded homogeneous domain in $\C^n$. For each $z \in D$, $\omega(z)$ and $\omega_0(z)$ are both finite.  Moreover, $\omega_0(z) \leq \omega(z) \leq \rho(z,0)$, where $\rho(z,w)$ is the Poincar\'e distance between $z$ and $w$.\end{lemma}

From Theorems 3.9 and 3.14 of \cite{Zhu:2004}, it follows that for $z \in \B_n$,
\begin{equation}\label{omegaequalityball}\omega_0(z) = \omega(z) = \rho(z,0) = \frac{1}{2}\log\frac{1+\norm{z}}{1-\norm{z}}.\end{equation}

It is not known if there are other bounded homogeneous domains for which equality holds.  However, for the unit polydisk it was shown in \cite{AllenColonna:2010} that for $z \in \D^n$, \begin{equation}\label{rhopolydisk}\rho(z,0) \leq \frac{1}{2}\sum_{j=1}^n \log\frac{1+\modu{z_j}}{1-\modu{z_j}}.\end{equation}

The quantities $\omega(z)$ and $\omega_0(z)$ play an important role in the theory of Bloch functions on bounded homogeneous domains, which is exercised through the use of the next Lemma.

\begin{lemma}[Lemma 4.2 of \cite{AllenColonna:2010}]\label{pointevaluationestimate} Let $D$ be a bounded homogeneous domain in $\C^n$ and let $f \in \Bloch(D)$ (respectively, $f \in \lilBlochDist(D)$).  Then for all $z \in D$, we have $$\modu{f(z)} \leq \modu{f(0)} + \omega(z)\beta_f,$$ (respectively, $\modu{f(z)} \leq \modu{f(0)} + \omega_0(z)\beta_f)$.
\end{lemma}

\begin{lemma}[Theorem 3.1 of \cite{AllenColonna:2009}]\label{bergmancharacterization} Let $D$ be a bounded homogeneous domain in $\C^n$ and let $f:D \to \C$ be holomorphic.  Then $f$ is Bloch if and only if  $f$ is a Lipschitz map as a function from $D$ under the Poincar\'e metric $\rho$ and the complex plane under the Euclidean metric.  Furthermore $$\beta_f = \sup_{z \neq w}\; \frac{\modu{f(z)-f(w)}}{\rho(z,w)}.$$
\end{lemma}

In particular, we have the following corollary.

\begin{corollary}\label{bergmanestimate} Let $D$ be a bounded homogeneous domain in $\C^n$ and $f \in \Bloch(D)$.  Then for all $z,w \in D$, $$\modu{f(z)-f(w)} \leq \norm{f}_\Bloch\rho(z,w).$$
\end{corollary}

\begin{lemma}[Theorem 3.3 of \cite{AllenColonna:2009}]\label{convergenceintoBloch} Let $(f_n)$ be a sequence of Bloch functions on a bounded homogeneous domain $D$ in $\C^n$ which converges locally uniformly in $D$ to some holomorphic function $f$.  If the sequence $(\beta_{f_n})$ is bounded, then $f$ if Bloch and $$\beta_f \leq \liminf_{n \to \infty} \beta_{f_n}.$$  That is, the function $f \mapsto \beta_f$ is lower semi-continuous on $\Bloch$ under the topology of uniform convergence on compact subsets of $D$.
\end{lemma}

\section{Boundedness and Operator Norm}\label{Section:Boundedness}
In this section, we characterize the bounded weighted composition operators, and determine the operator norm, from $\Bloch(D)$ to $\Hinfmu(D)$ in terms of the following quantities.  For $\psi \in H(D)$ and $\varphi \in S(D)$, define
$$\begin{aligned}
\upsilon_\mu(\psi,\varphi) &= \sup_{z \in D}\; \mu(z)\modu{\psi(z)}\omega(\varphi(z)), \text{ and}\\
\upsilon_{0,\mu}(\psi,\varphi) &= \sup_{z \in D}\; \mu(z)\modu{\psi(z)}\omega_0(\varphi(z)).
\end{aligned}$$

\begin{lemma}\label{upsilon inequality} Let $D$ be a bounded homogeneous domain in $\C^n$, $\psi \in H(D)$, and $\varphi \in S(D)$.  If $\wco:\Bloch(D)\to\Hinfmu(D)$ is bounded, then $$\upsilon_{0,\mu}(\psi,\varphi) \leq \upsilon_\mu(\psi,\varphi) \leq \norm{\wco}.$$
\end{lemma}

\begin{proof} The first inequality is a corollary of Lemma \ref{omega inequality}.  So it suffices to show that $\upsilon_\mu(\psi,\varphi) \leq \norm{\wco}$.  Let $f \in \Bloch(D)$ with $\blochnorm{f} \leq 1$.  For every $z \in D$, we have
$$\norm{\wco} \geq \munorm{\psi(f\circ \varphi)} \geq \mu(z)\modu{\psi(z)}\modu{f(\varphi(z))}.$$  Taking the supremum over all such $f \in \Bloch(D)$ such that $f(0) = 0$, we obtain $$\mu(z)\modu{\psi(z)}\omega(\varphi(z)) \leq \norm{\wco}.$$  Finally, taking the supremum over all $z \in D$, we have $\upsilon_\mu(\psi,\varphi) \leq \norm{\wco}$, as desired.
\end{proof}

\begin{theorem}\label{boundednesstheorem} Let $D$ be a bounded homogeneous domain in $\C^n$, $\psi \in H(D)$, and $\varphi$ a holomorphic self-map of $D$.  Then
\begin{enumerate}
\item[\listitem{a}] $\wco:\Bloch(D)\to\Hinfmu(D)$ is bounded if and only if $\psi \in \Hinfmu(D)$ and $\upsilon_\mu(\psi,\varphi)$ is finite.  Furthermore, if $\wco$ is bounded, then $\norm{\wco} = \max\;\left\{\munorm{\psi},\upsilon_\mu(\psi,\varphi)\right\}.$
\item[\listitem{b}] $\wco:\lilBlochDist(D)\to\Hinfmu(D)$ is bounded if and only if $\psi \in \Hinfmu(D)$ and $\upsilon_{0,\mu}(\psi,\varphi)$ is finite.  Furthermore, if $\wco$ is bounded, then $\norm{\wco} = \max\;\left\{\munorm{\psi},\upsilon_{0,\mu}(\psi,\varphi)\right\}.$
\end{enumerate}
\end{theorem}

\begin{proof} We will prove (a), since the proof of (b) follows the same argument.  First, assume $\wco$ is bounded.  Since the constant function $\1 \in \Bloch(D)$, we have $\psi = \wco \1 \in \Hinfmu(D)$.  Also $\upsilon_\mu(\psi,\varphi)$ is finite by Lemma \ref{upsilon inequality}, which also implies \begin{equation}\label{norm lower bound}\max\left\{\munorm{\psi},\upsilon_\mu(\psi,\varphi)\right\} \leq \norm{\wco}.\end{equation}

Next, assume $\psi \in \Hinfmu(D)$ and $\upsilon_\mu(\psi,\varphi)$ is finite.  By Lemma \ref{pointevaluationestimate}, for $f \in \Bloch(D)$ with $\blochnorm{f} \leq 1$, and $z \in D$,
$$\begin{aligned}
\munorm{\wco f} &= \sup_{z \in D}\; \mu(z)\modu{\psi(z)}\modu{f(\varphi(z))}\\
&\leq \sup_{z \in D}\; \mu(z)\modu{\psi(z)}\left(\modu{f(0)} + \omega(\varphi(z))\beta_f\right)\\
&\leq \munorm{\psi}\modu{f(0)} + \upsilon_\mu(\psi,\varphi)\beta_f\\
&\leq \munorm{\psi}\left(\blochnorm{f}-\beta_f\right) + \upsilon_\mu(\psi,\varphi)\beta_f\\
&\leq \munorm{\psi}\blochnorm{f} + \left(\upsilon_\mu(\psi,\varphi)-\munorm{\psi}\right)\beta_f\\
&\leq \max\left\{\munorm{\psi},\upsilon_\mu(\psi,\varphi)\right\}\blochnorm{f}.
\end{aligned}$$
Thus, $\wco$ is bounded, and taking the supremum over all $f \in \Bloch(D)$ with $\blochnorm{f} \leq 1$, we obtain $$\norm{\wco} \leq \max\left\{\munorm{\psi},\upsilon_\mu(\psi,\varphi)\right\},$$ as desired.
\end{proof}

\section{Compactness}\label{Section:Compactness}
In this section, we characterize the compact weighted composition operators from $\Bloch(D)$ into $\Hinfmu(D)$.  First, we provide necessary and sufficient conditions in terms of the classical convergence of bounded sequences.  Lastly, we provide sufficient conditions for the operator to be compact in terms of the ``little-oh" condition induced by the symbols.

\begin{proposition}\label{compactnesscharacterization} Let $D$ be a bounded homogeneous domain in $\C^n$, $\psi \in H(D)$, and $\varphi \in S(D)$.  Then $\wco: \Bloch(D) \to \Hinfmu(D)$ is compact if and only if for every bounded sequence $(f_k)$ in $\Bloch(D)$ converging to 0 locally uniformly in $D$, the sequence $(\munorm{\psi(f_k\circ\varphi)})$ converges to 0 as $k \to \infty$.
\end{proposition}

\begin{proof}
First, suppose $\wco:\Bloch(D)\to\Hinfmu(D)$ is compact.  Let $(f_k)$ be a bounded sequence in $\Bloch(D)$ which converges to 0 locally uniformly in $D$.  By the compactness of $\wco$, the sequence $(\psi(f_k\circ\varphi))$ contains a subsequence which converges to some function $f \in H(D)$.  For $z \in D$, $\mu(z)\psi(z)f_k(\varphi(z)) \to 0$ as $k \to \infty$.  Hence $f$ is identically 0, and thus $(\munorm{\psi(f_k\circ\varphi)})$ converges to 0 as $k\to \infty$.

Conversely, suppose $(\munorm{\psi(f_k\circ\varphi)})$ converges to 0 as $k \to \infty$ for every bounded sequence $(f_k)$ in $\Bloch(D)$ converging to 0 locally uniformly in $D$.  Let $(g_k)$ be a bounded sequence in $\Bloch(D)$, and without loss of generality assume $\norm{g_k}_\Bloch \leq 1$ for each $k \in \N$.  To prove $\wco$ is compact, it suffices to show there exists a subsequence $(g_{k_j})$ for which $(\psi(g_{k_j}\circ\varphi))$ converges in $\Hinfmu(D)$.  Fix $z_0 \in D$, and without loss of generality assume $g_k(z_0) = 0$ for all $k \in \N$.  By Corollary \ref{bergmanestimate}, $\modu{g_k(z)} \leq \rho(z,z_0)$ for all $z \in D$.  Thus $(g_k)$ is uniformly bounded on every closed disk centered at $z_0$ in the Poincar\'e metric, and thus is uniformly bounded on compact subsets of $D$.  By Montel's Theorem (see \cite{Scheidemann:2005}), there exists a subsequence $(g_{k_j})$ which converges locally uniformly to a function $g \in H(D)$.  By Lemma \ref{convergenceintoBloch}, $g \in \Bloch(D)$ with $\blochnorm{g} \leq 1$.  Defining $h_{k_j} = g_{k_j} - g$, we see that $\blochnorm{h_{k_j}} \leq 2$, and thus $(h_{k_j})$ is a bounded sequence in $\Bloch(D)$ converging to 0 locally uniformly in $D$.  Thus $(\munorm{\psi(h_{k_j}\circ\varphi)})$ converges to 0 as $k \to \infty$.  Therefore $(\psi(g_{k_j}\circ\varphi))$ converges to $\psi(g\circ\varphi)$ in $\Hinfmu(D)$.
\end{proof}

Although this is a characterization of the compact weighted composition operators from $\Bloch(D)$ into $\Hinfmu(D)$, we wish to have such a characterization in terms of the symbols $\psi$ and $\varphi$.  For a general bounded homogeneous domain, we can determine sufficient conditions for compactness.  We are unable to obtain a complete characterization due to the presence of the $\omega(z)$ term.  For the unit ball, we have a closed form for $\omega(z)$, and thus a complete characterization can be obtained.  However, in general, this term can not be removed, and thus we obtain a roadblock to obtaining complete characterizations of the weighted composition operators on bounded homogeneous domains.

\begin{theorem}\label{compactnesssufficiency} Let $D$ be a bounded homogeneous domain in $\C^n$, $\psi \in H(D)$, and $\varphi \in S(D)$.  Then $\wco:\Bloch(D)\to\Hinfmu(D)$ is compact if $\psi \in \Hinfmu(D)$ and $$\lim_{\varphi(z) \to \partial D} \frac{1}{2}\mu(z)\modu{\psi(z)}\omega(\varphi(z)) = 0.$$
\end{theorem}

\begin{proof}
Suppose $\psi \in \Hinfmu(D)$ and $\lim_{\varphi(z) \to \partial D} \frac{1}{2}\mu(z)\modu{\psi(z)}\omega(\varphi(z)) = 0$.  It suffices to show that for any bounded sequence $(f_k)$ in $\Bloch(D)$ converging to 0 locally uniformly in $D$, then $(\munorm{\psi(f_k\circ\varphi)})$ converges to 0 as $k \to \infty$.  Let $(f_k)$ be such a sequence, and without loss of generality assume $f_k(0) = 0$ for all $k \in \N$.  For $\varepsilon > 0$, there exists $r > 0$ such that $\mu(z)\modu{\psi(z)}\omega(\varphi(z)) < \varepsilon$ whenever $\rho(\varphi(z),\partial D) < r$.  So, if $\rho(\varphi(z),\partial D) < r$, then for $k$ large enough, $$\frac{1}{2}\mu(z)\modu{\psi(z)}\modu{f_k(\varphi(z))} \leq \frac{1}{2}\mu(z)\modu{\psi(z)}\omega(\varphi(z)) < \varepsilon.$$

Since $(f_k)$ converges to 0 locally uniformly in $D$, if $\rho(\varphi(z),\partial D) \geq r$ then $\modu{f_k(\varphi(z))} \to 0$ as $k \to \infty$.  So for $k$ large enough, $\modu{f_k(\varphi(z))} < \frac{\varepsilon}{2\munorm{\psi}}$ and $$\frac{1}{2}\mu(z)\modu{\psi(z)}\modu{f_k(\varphi(z))} \leq\ \munorm{\psi}\modu{f_k(\varphi(z))} < \varepsilon,$$ whenever $\rho(\varphi(z),\partial D) \geq r$.  Thus, for $k$ large enough, $\frac{1}{2}\mu(z)\modu{\psi(z)}\modu{f_k(\varphi(z))} < \varepsilon$ for all $z \in D$.  Therefore, $(\munorm{\psi(f_k\circ\varphi)})$ converges to 0 as $k \to \infty$, completing the proof.
\end{proof}

As was the case for such operators acting on the Bloch space \cite{AllenColonna:2010}, more information about $\omega(z)$ on these domains is needed to prove the necessity of these conditions.  However, we believe this to be the case, as the next section offers some hope.  We end this section with the following conjecture.

\begin{conjecture}\label{conjecture} Let $D$ be a bounded homogeneous domain in $\C^n$, $\psi \in H(D)$, and $\varphi$ a holomorphic self-map of $D$.  Then $\wco:\Bloch(D)\to\Hinfmu(D)$ is compact if and only if $\psi \in \Hinfmu(D)$ and $$\lim_{\varphi(z) \to \partial D}\frac{1}{2} \mu(z)\modu{\psi(z)}\omega(\varphi(z)) = 0.$$\end{conjecture}

\section{On the Unit Polydisk}\label{Section:Polydisk}
In this section, we apply the results of the previous sections to the case when the bounded homogeneous domain is the unit polydisk $\D^n$.  In this case, we show the sufficient condition for compactness in Theorem \ref{compactnesssufficiency} is also necessary.  The weighted composition operator has not been studied in this setting before, and so the results in this section are new to the literature.

To characterize the bounded weighted composition operators completely in terms of the symbols, we must establish a closed form for $\omega(z)$ on the domain.  As stated previously, this is known for the unit ball.  However, for the unit polydisk, such a form is not known.  Thus, a literal application of the results from the previous two sections does not yield a characterization of boundedness completely in terms of the symbols.   The same can be said for the operator norm and the sufficient condition for compactness.  In this section, we will provide characterizations of boundedness and compactness completely in terms of the symbols, as well as provide norm estimates in a similar manner.  For $\psi \in H(\D^n)$ and $\varphi = (\varphi_1,\dots,\varphi_n) \in S(\D^n)$, define $$\vartheta_\mu(\psi,\varphi) = \sup_{z \in \D^n} \frac{1}{2}\mu(z)|\psi(z)|\sum_{j=1}^n\log\frac{1+|\varphi_j(z)|}{1-|\varphi_j(z)|}.$$

\begin{theorem}\label{boundednesstheorempolydisk} Let $\psi \in H(\D^n)$ and $\varphi=(\varphi_1,\dots,\varphi_n) \in S(D^n)$.  Then the following are equivalent:
\begin{enumerate}
\item[\listitem{a}] $\wco:\Bloch(\D^n)\to\Hinfmu(\D^n)$ is bounded.
\item[\listitem{b}] $\wco:\lilBlochDist(\D^n)\to\Hinfmu(\D^n)$ is bounded.
\item[\listitem{c}] $\psi \in \Hinfmu(\D^n)$ and $\vartheta_\mu(\psi,\varphi)$ is finite.
\end{enumerate}
\end{theorem}

\begin{proof}
The implication $(a) \Rightarrow (b)$ is clear.  So assume $\wco:\lilBlochDist(\D^n) \to \Hinfmu(\D^n)$ is bounded.  Then $\psi = \wco\1 \in \Hinfmu(\D^n)$.  Fix $j \in \{1,\dots,n\}$ and $\lambda \in \D^n$.  Then the function $$h_j(z) = \frac{1}{n(2+\log 4)}\Log\frac{4}{1-z_j\overline{\varphi_j(\lambda)}}$$ is in $\lilBlochDist(\D^n)$ with $\blochnorm{h_j} \leq \frac{1}{n}$ (see \cite{Allen:2009}).  Define $$f(z) = \sum_{j=1}^n h_j(z) = \frac{1}{n(2+\log 4)}\sum_{j=1}^n \Log\frac{4}{1-z_j\overline{\varphi_j(\lambda)}}.$$  Then $f \in \lilBlochDist(\D^n)$ with $\blochnorm{f} \leq \sum_{j=1}^n \blochnorm{h_j} \leq 1.$  Since $\wco$ is bounded, we have
$$\begin{aligned}
\norm{\wco} &\geq \munorm{\wco f}\\
&\geq \mu(\lambda)\modu{\psi(\lambda)}\modu{f(\varphi(\lambda))}\\
&= \frac{2}{n(2+\log 4)}\frac{1}{2}\mu(\lambda)\modu{\psi(\lambda)}\sum_{j=1}^n\log\frac{4}{1-\modu{\varphi_j(\lambda)}^2}\\
&\geq \frac{1}{n(1+\log 2)} \frac{1}{2}\mu(\lambda)\modu{\psi(\lambda)}\sum_{j=1}^n\log\frac{1+\modu{\varphi_j(\lambda)}}{1-\modu{\varphi_j(\lambda)}}.
\end{aligned}$$  Taking the supremum over all $\lambda \in \D^n$, we obtain \begin{equation}\label{newopbound}\vartheta_\mu(\psi,\varphi) = \sup_{\lambda \in \D^n}\; \frac{1}{2}\mu(\lambda)\modu{\psi(\lambda)}\sum_{j=1}^n\log\frac{1+\modu{\varphi_j(\lambda)}}{1-\modu{\varphi_j(\lambda)}} \leq n(1+\log 2)\norm{\wco},\end{equation} which is finite since $\wco$ is assumed to be bounded.

Finally, suppose $\psi \in \Hinfmu(\D^n)$ and $\vartheta_\mu(\psi,\varphi)$ is finite.  By relation (\ref{rhopolydisk}), we have $$\upsilon_\mu(\psi,\varphi) = \sup_{z \in \D^n}\; \frac{1}{2}\mu(z)\modu{\psi(z)}\omega(\varphi(z)) \leq \sup_{z \in \D^n}\; \mu(z)\modu{\psi(z)}\sum_{j=1}^n\log\frac{1+\modu{\varphi_j(z)}}{1-\modu{\varphi_j(z)}} = \vartheta_\mu(\psi,\varphi),$$ which is finite.  Thus by Theorem \ref{boundednesstheorem}, $\wco$ is bounded from $\Bloch(\D^n)$ to $\Hinfmu(\D^n)$, completing the proof.
\end{proof}

As mentioned before, without an equivalent form for $\omega(z)$ on $\D^n$, a norm equality purely in terms of the symbols can not be established.  The norm estimates we provide follow immediately from relations (\ref{omegaequalityball}), (\ref{rhopolydisk}), (\ref{newopbound}), and Theorem \ref{boundednesstheorem}.

\begin{theorem}\label{polynorm} Let $\psi \in H(\D^n)$ and $\varphi=(\varphi_1,\dots,\varphi_n) \in S(\D^n)$.  If $W_{\psi,\varphi}:\Bloch(\D^n) \to \Hinfmu(\D^n)$ is bounded, then $$\max\;\left\{\|\psi\|_{\Hinfmu}, \frac{1}{n(1+\log 2)}\vartheta_\mu(\psi,\varphi)\right\} \leq \|W_{\psi,\varphi}\| \leq \max\left\{\|\psi\|_{\Hinfmu},\vartheta_\mu(\psi,\varphi)\right\}.$$
\end{theorem}
Finally, we show the sufficient condition for compactness from Theorem \ref{compactnesssufficiency} is necessary.

\begin{theorem}\label{compactnesstheorem} Let $\psi \in H(\D^n)$ and $\varphi=(\varphi_1,\dots,\varphi_n) \in S(\D^n)$.  Then the following are equivalent:
\begin{enumerate}
\item[\listitem{a}] $\wco:\Bloch(\D^n)\to\Hinfmu(\D^n)$ is compact.
\item[\listitem{b}] $\wco:\lilBlochDist(\D^n)\to\Hinfmu(\D^n)$ is compact.
\item[\listitem{c}] $\psi \in \Hinfmu(\D^n)$ and $$\lim_{\varphi(z)\to\partial \D^n} \frac{1}{2}\mu(z)\modu{\psi(z)}\sum_{j=1}^n\log\frac{1+\modu{\varphi_j(z)}}{1-\modu{\varphi_j(z)}} =0.$$
\end{enumerate}
\end{theorem}

\begin{proof}
The implication $(a) \Rightarrow (b)$ is clear.  Suppose $\wco:\lilBlochDist(\D^n) \to \Hinfmu(\D^n)$ is compact.  So by the boundedness of $\wco$ and Theorem \ref{boundednesstheorempolydisk}, $\psi \in \Hinfmu(\D^n)$ and \begin{equation}\label{supboundpolydisk}\sup_{z \in \D^n}\;\frac{1}{2}\mu(z)\modu{\psi(z)}\sum_{j=1}^n\log\frac{1+\modu{\varphi_j(z)}}{1-\modu{\varphi_j(z)}} < \infty.\end{equation}  Let $(z^{(k)})$ be a sequence in $\D^n$ such that $\varphi(z^{(k)}) \to \partial\D^n$ as $k \to \infty$.  So there is an index $m \in \{1,\dots,n\}$ such that $\modu{\varphi_m(z^{(k)})} \to 1$ as $k \to \infty$.  It follows that $$\sum_{j=1}^n \log\frac{1+\modu{\varphi_j(z^{(k)})}}{1-\modu{\varphi_j(z^{(k)})}} \to \infty$$ as $k \to \infty$.  So by inequality (\ref{supboundpolydisk}), it must be the case that $$\lim_{k \to \infty} \mu(z^{(k)})\psi(z^{(k)}) = 0.$$

For $k \in \N$, define $$f_k(z) = \displaystyle\frac{\left(\Log\displaystyle\frac{4}{1-z_m\conj{\varphi_m(z^{(k)})}}\right)^2}{\log\displaystyle\frac{4}{1-\modu{\varphi_m(z^{(k)})}^2}}.$$ Then the sequence $(f_k)$ is bounded in $\lilBlochDist(\D^n)$ and converges to 0 locally uniformly in $\D^n$ (see \cite{Allen:2009}).  By the compactness of $\wco$, we have
\begin{eqnarray}
\notag \frac{1}{2}\mu(z^{(k)})\modu{\psi(z^{(k)})}\log\frac{1+\modu{\varphi_m(z^{(k)})}}{1-\modu{\varphi_m(z^{(k)})}} &\leq& \frac{1}{2}\mu(z^{(k)})\modu{\psi(z^{(k)})}\log\frac{4}{1-\modu{\varphi_m(z^{(k)})}^2}\\
\notag &=& \frac{1}{2}\mu(z^{(k)})\modu{\psi(z^{(k)})f_k(\varphi(z^{(k)}))}\\
\label{polyestimate2}&\leq& \munorm{\psi(f_k\circ\varphi)} \to 0
\end{eqnarray}
as $k \to \infty$.

Now let $\ell \in \{1,\dots,n\}$ such that $\modu{\varphi_\ell(z^{(k)})} \not\to 1$ as $k \to \infty$.  So there exists $r \in (0,1)$ such that $\modu{\varphi_\ell(z^{(k)})}\leq r$ for all $k \in \N$.  Since $\mu(z^{(k)})\psi(z^{(k)}) \to 0$ as $k \to \infty$, we have
\begin{equation}\label{polyestimate3}\frac{1}{2}\mu(z^{(k)})\modu{\psi(z^{(k)})}\log\frac{1+\modu{\varphi_\ell(z^{(k)})}}{1-\modu{\varphi_\ell(z^{(k)})}} \leq \frac{1}{2}\mu(z^{(k)})\modu{\psi(z^{(k)})}\log\frac{1+r}{1-r} \to 0\end{equation} as $k \to \infty$.  So by inequalities (\ref{polyestimate2}) and (\ref{polyestimate3}), we have $$\lim_{k \to \infty} \frac{1}{2}\mu(z^{(k)})\modu{\psi(z^{(k)})}\sum_{j=1}^n\log\frac{1+\modu{\varphi_j(z^{(k)}))}}{1-\modu{\varphi_j(z^{(k)}))}} = 0,$$ showing $(b) \Rightarrow (c)$.

Lastly, suppose $\psi \in \Hinfmu(\D^n)$ and $$\lim_{\varphi(z)\to\partial \D^n} \frac{1}{2}\mu(z)\modu{\psi(z)}\sum_{j=1}^n\log\frac{1+\modu{\varphi_j(z)}}{1-\modu{\varphi_j(z)}} =0.$$  Then from relation (\ref{rhopolydisk}), we obtain
$$\lim_{\varphi(z)\to\partial\D^n} \frac{1}{2}\mu(z)\modu{\psi(z)}\omega(\varphi(z)) \leq \lim_{\varphi(z)\to\partial\D^n}\frac{1}{2}\mu(z)\modu{\psi(z)}\sum_{j=1}^n\log\frac{1+\modu{\varphi(z)}}{1-\modu{\varphi(z)}} = 0.$$  Therefore, by Theorem \ref{compactnesssufficiency}, $\wco$ is compact, proving $(c) \Rightarrow (a)$.
\end{proof}

\section{Conclusions and Open Questions}\label{Section:Conclusions}
The general bounded homogeneous domains offer a framework by which traditional operator theory on spaces of holomorphic functions in several complex variables (specifically on the unit ball and the unit polydisk) can be unified.  As with any generalization, any results we gain are met with trade-offs.  One of the trade-offs is the use of the $\omega(z)$ in characterizing boundedness and compactness of the multiplication, composition, and weighted composition operators.  For the unit ball, a closed form for $\omega(z)$ is known, and upper bounds are known for the unit polydisk.  However, no closed form for the polydisk, or no bounds known for other bounded homogeneous domains presents a temporary roadblock.  This roadblock presents itself when we consider conditions for compactness of the composition and weighted composition operators.  

We conclude this paper with the following open questions we hope researchers will take up.
\begin{enumerate}
\item[\listitem{1}] Is there a closed form for $\omega(z)$ on $\D^n$?  Is there a sharper bound than $\rho(z,0)$?
\item[\listitem{2}] Are there other specific bounded homogeneous domains for which a closed form for $\omega(z)$ can be determined?
\item[\listitem{3}] Is Conjecture \ref{conjecture} true for all bounded homogeneous domains?  Is Conjecture \ref{conjecture} true for any other specific bounded homogeneous domain?
\end{enumerate}

\bibliographystyle{amsplain}
\bibliography{references.bib}
\end{document}